\documentclass{article}%
\usepackage{amsmath}
\usepackage{amsfonts}
\usepackage{amssymb}
\usepackage{graphicx}%
\newtheorem{theorem}{Theorem}
\newtheorem{acknowledgement}[theorem]{Acknowledgement}

\newtheorem{claim}[theorem]{Claim}

\newtheorem{corollary}[theorem]{Corollary}

\newtheorem{definition}[theorem]{Definition}

\newtheorem{lemma}[theorem]{Lemma}

\newtheorem{problem}[theorem]{Problem}
\newtheorem{proposition}[theorem]{Proposition}

\newenvironment{proof}[1][Proof]{\noindent\textbf{#1.} }{\ \rule{0.5em}{0.5em}}
\begin{document}

\date{}
\title{There is a $+$-Ramsey \textsf{MAD} family}
\author{Osvaldo Guzm\'{a}n \thanks{This research forms part of the the
author Ph.D thesis, which was supported by CONACyT scholarship 420090{}.
\textit{keywords: }MAD families, $+$-Ramsey MAD families, destructibility,
Miller forcing, cardinal invariants. \textit{AMS Classification: }03E17, 03E05
, 03E20}}
\maketitle

\begin{abstract}
We answer an old question of Michael Hru\v{s}\'{a}k by constructing a
$+$-Ramsey \textsf{MAD} family without the need of any additional axioms
beyond $\mathsf{ZFC.}$ We also prove that every Miller-indestructible
\textsf{MAD }family is $+$-Ramsey, this improves a result of Michael Hru\v{s}\'{a}k.

\end{abstract}

\section*{Introduction}

A family $\mathcal{A}\subseteq\left[  \omega\right]  ^{\omega}$ is
\emph{almost disjoint (AD) }if the intersection of any two different elements
of $\mathcal{A}$ is finite, a \textsf{MAD }\emph{family}\textsf{\emph{ }}is
a\textsf{ }maximal almost disjoint family. Almost disjoint families and
\textsf{MAD} families have become very important in set theory, topology and
functional analysis (see \cite{AlmostDisjointFamiliesandTopology}). It is very
easy to prove that the Axiom of Choice implies the existence of \textsf{MAD}
families. However, constructing \textsf{MAD} families with special
combinatorial or topological properties is a very difficult task without the
an additional hypothesis beyond \textsf{ZFC}. Constructing models of set
theory for which certain kinds of \textsf{MAD} families do not exist is very
difficult. We would like to mention some important examples regarding the
existence or non-existence of special \textsf{MAD} families:

\begin{enumerate}
\item (Simon \cite{ACompactFrechetSpacewhoseSquareisnotFrechet}) There is a
\textsf{MAD} family which can be partitioned into two nowhere \textsf{MAD} families.

\item (Mr\'{o}wka \cite{Mrowka}) There is a \textsf{MAD} family for which its
$\Psi$-space has a unique compactification.

\item (Raghavan \cite{ThereisavanDouwenMADfamily}) There is a van Douwen
\textsf{MAD} family.

\item (Raghavan \cite{AModelwithnoStronglySeparableAlmostDisjointFamilies})
There is a model with no strongly separable \textsf{MAD} families.
\end{enumerate}

\qquad\ \ \ \qquad\ \ \qquad\ \ \ \ \ \ 

If $\mathcal{I}$ is an ideal on $\omega$ then by $\mathcal{I}^{+}$ we denote
the set$\wp\left(  \omega\right)  $ $\backslash\mathcal{I}$ and its elements
are called $\mathcal{I}$-positive sets. If $\mathcal{A}$ is an \textsf{AD}
family by $\mathcal{I}\left(  \mathcal{A}\right)  $ we denote the ideal
generated by $\mathcal{A}.$ In \cite{HappyFamilies} Adrian Mathias proved that
if $\mathcal{A}$ is a \textsf{MAD} family then $\mathcal{I}\left(
\mathcal{A}\right)  ^{+}$ is a happy family, which is a kind of Ramsey-like
property. In \cite{SelectivityofAlmostDisjointFamilies} Michael Hru\v{s}\'{a}k
introduced a stronger Ramsey property:

\begin{definition}
\qquad\ \ 

\begin{enumerate}
\item By $\mathcal{A}^{\perp}$ we denote the set of all $X\subseteq\omega$
such that $\mathcal{A}\cup\left\{  X\right\}  $ is almost disjoint.

\item If $\mathcal{A}$ is an \textsf{AD }family, by $\mathcal{I}\left(
\mathcal{A}\right)  ^{++}$ we denote the set of all $X\subseteq\omega$ such
that there is $\mathcal{B}\in\left[  \mathcal{A}\right]  ^{\omega}$ such that
$\left\vert X\cap B\right\vert =\omega$ for every $B\in\mathcal{B}.$

\item Let $\mathcal{X}\subseteq\left[  \omega\right]  ^{\omega},$ we say a
tree $T\subseteq\omega^{<\omega}$ is a $\mathcal{X}$\emph{-branching tree} if
$suc_{T}\left(  s\right)  \in\mathcal{X}$ for every $s\in T$ (where
$suc_{T}\left(  s\right)  =\left\{  n\in\omega\mid s^{\frown}\left\langle
n\right\rangle \in T\right\}  $).

\item An \textsf{AD} family $\mathcal{A}$ is $+$\emph{-Ramsey} if for every
$\mathcal{I}\left(  \mathcal{A}\right)  ^{+}$-branching tree $T,$ there is
$f\in\left[  T\right]  $ such that $im\left(  f\right)  \in\mathcal{I}\left(
\mathcal{A}\right)  ^{+}.$
\end{enumerate}
\end{definition}

\qquad\ \ \ \qquad\ \ \qquad\ \ 

In \cite{SelectivityofAlmostDisjointFamilies} it is proved that there is a
\textsf{MAD} family that is not $+$-Ramsey. On the other hand, $+$-Ramsey
\textsf{MAD} families can be contructed under $\mathfrak{b=c},$ \textsf{cov}%
$\left(  \mathcal{M}\right)  =\mathfrak{c,}$ $\mathfrak{a}<$ \textsf{cov}%
$\left(  \mathcal{M}\right)  $ or $\Diamond\left(  \mathfrak{b}\right)  $ (see
\cite{SelectivityofAlmostDisjointFamilies} and
\cite{OrderingMADfamiliesalaKatetov}). Michael Hru\v{s}\'{a}k asked the
following,\qquad\ \ 

\begin{problem}
[Hru\v{s}\'{a}k \cite{SelectivityofAlmostDisjointFamilies}]Is there a
$+$-Ramsey \textsf{MAD} family in \textsf{ZFC}?
\end{problem}

\qquad\ \ \ 

In this note we will provide a positive answer to this question. In
\cite{SANEPlayer} (see also \cite{AlmostDisjointFamiliesandTopology} and
\cite{SplittingFamiliesandCompleteSeparability}) Saharon Shelah developed a
novel and powerfull method to construct \textsf{MAD} families. He used it to
prove that there is a completely separable \textsf{MAD} family if
$\mathfrak{s\leq a}$ or $\mathfrak{a<s}$ and a certain \textsf{PCF}-hypothesis
holds. Our technique for constructing a $+$-Ramsey \textsf{MAD} is based on
the technique of Shelah (however, in this case we were able to avoid the
\textsf{PCF}-hypothesis). It is worth mentioning that the method of Shelah has
been further developed in \cite{OnWeaklytightFamilies} and
\cite{SplittingFamiliesandCompleteSeparability} where it is proved that weakly
tight \textsf{MAD} families exist under $\mathfrak{s\leq b.}$ Our notation is
mostly standard, the definition and basic properties of the cardinal
invariants of the continuum used in this note can be found in
\cite{HandbookBlass}.

\section*{Preliminaries}

A \textsf{MAD} family $\mathcal{A}$ is \emph{completely separable }if for
every $X\in\mathcal{I}\left(  \mathcal{A}\right)  ^{+}$ there is
$A\in\mathcal{A}$ such that $A\subseteq X.$ This type of \textsf{MAD} families
was introduced by Hechler in \cite{Hechler}. A year later, Shelah and
Erd\"{o}s asked the following question:

\begin{problem}
[Erd\"{o}s-Shelah]Is there a completely separable \textsf{MAD} family?
\end{problem}

\qquad\ \ \ \qquad\ \ \ 

It is easy to construct models where the previous question has a positive
answer. It was shown by Balcar and Simon (see \cite{DisjointRefinement}) that
such families exist assuming one of the following axioms: $\mathfrak{a=c},$
$\mathfrak{b=d},$ $\mathfrak{d\leq a}$ and $\mathfrak{s}=\omega_{1}.$ In
\cite{SANEPlayer} (see also \cite{AlmostDisjointFamiliesandTopology} and
\cite{SplittingFamiliesandCompleteSeparability}) Shelah developed a novel and
powerful method to construct completely separable \textsf{MAD} families. He
used it to prove that there are such families if either $\mathfrak{s\leq a}$
or $\mathfrak{a<s}$ and a certain (so called) \textsf{PCF }hypothesis holds
(which holds for example, if the continuum is less than $\aleph_{\omega}$).
Since the construction of Shelah of a completely separable \textsf{MAD} family
under $\mathfrak{s\leq a}$ is key for our construction of a $+$-Ramsey
\textsf{MAD} family, we will recall it on this section. This exposition is
based on \cite{SplittingFamiliesandCompleteSeparability} and
\cite{AlmostDisjointFamiliesandTopology}.

\begin{definition}
\qquad\ \ \qquad\ \ 

\begin{enumerate}
\item We say that $S$ \emph{splits }$X$ if $S\cap X$ and $X\setminus S$ are
both infinite.

\item $\mathcal{S\subseteq}$ $\left[  \omega\right]  ^{\omega}$ is a
\emph{splitting family }if for every $X\in\left[  \omega\right]  ^{\omega}$
there is $S\in\mathcal{S}$ such that $S$ splits $X.$

\item Let $S\in\left[  \omega\right]  ^{\omega}$ and $\mathcal{P}=\left\{
P_{n}\mid n\in\omega\right\}  $ be an interval partition. We say $S$
\emph{block-splits }$\mathcal{P}$ if both of the sets $\left\{  n\mid
P_{n}\subseteq S\right\}  $ and $\left\{  n\mid P_{n}\cap S=\emptyset\right\}
$ are infinite.

\item A family $\mathcal{S}\subseteq\left[  \omega\right]  ^{\omega}$ is
called a \emph{block-splitting family }if every interval partition is
block-split by some element of $\mathcal{S}.$
\end{enumerate}
\end{definition}

\qquad\ \ \ 

Recall that the \emph{splitting number }$\mathfrak{s}$ is the smallest size of
a splitting family. It is well known that $\mathfrak{s}$ has uncountable
cofinality, it is below the dominating number $\mathfrak{d}$ and independent
from the unbounding number $\mathfrak{b}$ (see\cite{HandbookBlass}). Regarding
the smallest size of a block splitting family we have the following result of
Kamburelis and Weglorz:

\begin{proposition}
[\cite{Splittings}]The smallest size of a block-splitting family is
$max\left\{  \mathfrak{b},\mathfrak{s}\right\}  .$
\end{proposition}

\qquad\ \ 

Some other notions of splitting are the following:

\begin{definition}
Let $S\in\left[  \omega\right]  ^{\omega}$ and $\overline{X}=\left\{
X_{n}\mid n\in\omega\right\}  \subseteq\left[  \omega\right]  ^{\omega}$.

\begin{enumerate}
\item We say that $S$ $\omega$\emph{-splits} $\overline{X}$ if $S$ splits
every $X_{n}.$

\item We say that $S$ $\left(  \omega,\omega\right)  $\emph{-splits}
$\overline{X}$ if both the sets $\left\{  n\mid\left\vert X_{n}\cap
S\right\vert =\omega\right\}  $ and $\left\{  n\mid\left\vert X_{n}\cap\left(
\omega\backslash S\right)  \right\vert =\omega\right\}  $ are infinite.

\item We say that $\mathcal{S}\subseteq\left[  \omega\right]  ^{\omega}$ is an
$\omega$\emph{-splitting family }if every countable collection of infinite
subsets of $\omega$ is $\omega$-split by some element of $S.$

\item We say that $\mathcal{S}\subseteq\left[  \omega\right]  ^{\omega}$ is an
$\left(  \omega,\omega\right)  $\emph{-splitting family }if every countable
collection of infinite subsets of $\omega$ is $\left(  \omega,\omega\right)
$-split by some element of $S.$
\end{enumerate}
\end{definition}

\qquad\ \ \qquad\ \ 

It is easy to see that every block splitting family is an $\omega$-splitting
family. The following is a fundamental result of Mildenberger, Raghavan and Stepr\={a}ns:

\begin{proposition}
[\cite{SplittingFamiliesandCompleteSeparability}]There is an $\left(
\omega,\omega\right)  $-splitting family of size $\mathfrak{s}.$
\end{proposition}

\qquad\ \ 

The key combinatorial feature of $\left(  \omega,\omega\right)  $-splitting
families is the following result of Raghavan and Stepr\={a}ns:

\begin{proposition}
[\cite{OnWeaklytightFamilies}]If $\mathcal{S}$ is an $\left(  \omega
,\omega\right)  $-splitting family, $\mathcal{A}$ an \textsf{AD} family and
$X\in\mathcal{I}\left(  \mathcal{A}\right)  ^{+}$ then there is $S\in
\mathcal{S}$ such that $X\cap S,$ $X\cap\left(  \omega\setminus S\right)
\in\mathcal{I}\left(  \mathcal{A}\right)  ^{+}.$
\end{proposition}

\qquad\ \ \ 

Given $X\subseteq\omega$ we denote $X^{0}=X$ and $X^{1}=\omega\setminus X.$ By
the previous result, if $\mathcal{A}$ is an \textsf{AD} family, $X\in
\mathcal{I}\left(  \mathcal{A}\right)  ^{+}$ and $\mathcal{S}=\left\{
S_{\alpha}\mid\alpha<\mathfrak{s}\right\}  $ is an $\left(  \omega
,\omega\right)  $-splitting family then there are $\alpha<\mathfrak{s}$ and
$\tau_{X}^{\mathcal{A}}\in2^{\alpha}$ such that:

\begin{enumerate}
\item If $\beta<\alpha$ then $X\cap S_{\beta}^{1-\tau_{X}^{\mathcal{A}}\left(
\beta\right)  }\in\mathcal{I}\left(  \mathcal{A}\right)  .$

\item $X\cap S_{\alpha},$ $X\setminus S_{\alpha}\in\mathcal{I}\left(
\mathcal{A}\right)  ^{+}.$
\end{enumerate}

\qquad\ \ \ \ 

Clearly $\tau_{X}^{\mathcal{A}}$ $\in2^{<\mathfrak{s}}$ is unique and if
$Y\in\left[  X\right]  ^{\omega}\cap\mathcal{I}\left(  \mathcal{A}\right)
^{+}$ then $\tau_{Y}^{\mathcal{A}}$ extends $\tau_{X}^{\mathcal{A}}.$ We can
now prove the main result of this section:\qquad\ \ \ 

\begin{theorem}
[Shelah \cite{SANEPlayer}]If $\mathfrak{s\leq a}$ then there is a completely
separable \textsf{MAD} family.
\end{theorem}

\begin{proof}
Let $\left[  \omega\right]  ^{\omega}=\left\{  X_{\alpha}\mid\alpha
<\mathfrak{c}\right\}  .$ We will recursively construct $\mathcal{A}=\left\{
A_{\alpha}\mid\alpha<\mathfrak{c}\right\}  $ and $\left\{  \sigma_{\alpha}%
\mid\alpha<\mathfrak{c}\right\}  \subseteq2^{<\mathfrak{s}}$ such that for
every $\alpha<\mathfrak{c}$ the following holds: (where $\mathcal{A}_{\alpha
}=\left\{  A_{\xi}\mid\xi<\alpha\right\}  $)

\begin{enumerate}
\item $\mathcal{A}_{\alpha}$ is an \textsf{AD} family.

\item If $X_{\alpha}\in\mathcal{I}\left(  \mathcal{A}_{\alpha}\right)  ^{+}$
then $A_{\alpha}\subseteq X_{\alpha}.$

\item If $\alpha\neq\beta$ then $\sigma_{\alpha}\neq\sigma_{\beta}.$

\item If $\xi<dom\left(  \sigma_{\alpha}\right)  $ then $A_{\alpha}%
\subseteq^{\ast}S_{\xi}^{\sigma_{\alpha}\left(  \xi\right)  }.$
\end{enumerate}

\qquad\ \ 

It is clear that if we manage to do this then we will have achieved to
construct a completely separable \textsf{MAD} family. Assume $\mathcal{A}%
_{\delta}=\left\{  A_{\xi}\mid\xi<\delta\right\}  $ has already been
constructed. Let $X=$ $X_{\delta}$ if $X_{\delta}\in\mathcal{I}\left(
\mathcal{A}_{\delta}\right)  ^{+}$ and if $X_{\delta}\in\mathcal{I}\left(
\mathcal{A}_{\delta}\right)  $ let $X$ be any other element of $\mathcal{I}%
\left(  \mathcal{A}_{\delta}\right)  ^{+}.$ We recursively find $\left\{
X_{s}\mid s\in2^{<\omega}\right\}  \subseteq\mathcal{I}\left(  \mathcal{A}%
_{\delta}\right)  ^{+}$ ,$\left\{  \eta_{s}\mid s\in2^{<\omega}\right\}
\subseteq2^{<\mathfrak{s}}$ and $\left\{  \alpha_{s}\mid s\in2^{<\omega
}\right\}  $ as follows:

\begin{enumerate}
\item $X_{\emptyset}=X.$

\item $\eta_{s}=\tau_{X_{s}}^{\mathcal{A}_{\delta}}$ and $\alpha
_{s}=dom\left(  \eta_{s}\right)  .$

\item $X_{s^{\frown}0}=X_{s}\cap S_{\alpha_{s}}$ and $X_{s^{\frown}1}%
=X_{s}\cap\left(  \omega\setminus S_{\alpha_{s}}\right)  .$
\end{enumerate}

\qquad\ \ 

Note that if $t\subseteq s$ then $X_{s}\subseteq X_{t}$ and $\eta_{t}%
\subseteq\eta_{s}.$ On the other hand, if $s$ is incompatible with $t$ then
$\eta_{s}$ and $\eta_{t}$ are incompatible. For every $f\in2^{\omega}$ let
$\eta_{f}=\bigcup\limits_{n\in\omega}\eta_{f\upharpoonright n}$. Since
$\mathfrak{s}$ has uncountable cofinality, each $\eta_{f}$ is an element of
$2^{<\mathfrak{s}}$ and if $f\neq g$ then $\eta_{f}$ and $\eta_{g}$ are
incompatible nodes of $2^{<\mathfrak{s}}.$ Since $\delta$ is smaller than
$\mathfrak{c}$ there is $f\in2^{\omega}$ such that there is no $\alpha<\delta$
such that $\sigma_{\alpha}$ extends $\eta_{f}.$ Since $\left\{
X_{f\upharpoonright n}\mid n\in\omega\right\}  $ is a decreasing sequence of
elements in $\mathcal{I}\left(  \mathcal{A}_{\delta}\right)  ^{+}$, there is
$Y\in\mathcal{I}\left(  \mathcal{A}_{\delta}\right)  ^{+}$ such that
$Y\subseteq^{\ast}X_{f\upharpoonright n}$ for every $n\in\omega$ (see
\cite{HappyFamilies} proposition 0.7 or
\cite{AlmostDisjointFamiliesandTopology} proposition 2).

\qquad\ \ \ \qquad\ \ \ \ \ 

Letting $\beta=dom\left(  \eta_{f}\right)  ,$ we claim that if $\xi<\beta$
then $Y\cap S_{\xi}^{1-\eta_{f}\left(  \xi\right)  }\in\mathcal{I}\left(
\mathcal{A}\right)  .$ To prove this, let $n$ be the first natural number such
that $\xi<dom\left(  \eta_{f\upharpoonright n}\right)  .$ By our construction,
we know that $X_{f\upharpoonright n}\cap S_{\xi}^{1-\eta_{f}\left(
\xi\right)  }\in\mathcal{I}\left(  \mathcal{A}\right)  $ and since
$Y\subseteq^{\ast}X_{f\upharpoonright n}$ the result follows.

\qquad\ \ \ 

For every $\xi<\beta$ let $F_{\xi}\in\left[  \mathcal{A}\right]  ^{<\omega}$
such that $Y\cap S_{\xi}^{1-\eta_{f}\left(  \xi\right)  }\subseteq^{\ast
}\bigcup F_{\xi}$ and let $W=\left\{  A_{\alpha}\mid\sigma_{\alpha}%
\subseteq\eta_{f}\right\}  .$ Let $\mathcal{D}=W\cup\bigcup\limits_{\xi<\beta
}F_{\xi}$ and note that $\mathcal{D}$ has size less than $\mathfrak{s},$ hence
it has size less than $\mathfrak{a}.$ In this way we conclude that
$Y\upharpoonright\mathcal{D}$ is not a \textsf{MAD} family, so there is
$A_{\delta}\in\left[  Y\right]  ^{\omega}$ that is almost disjoint with every
element of $\mathcal{D}$ and define $\sigma_{\delta}=\eta_{f}.$ We claim that
$A_{\delta}$ is almost disjoint with $\mathcal{A}_{\delta}.$ To prove this,
let $\alpha<\delta,$ in case $A_{\alpha}\in W$ we already know $A_{\alpha}\cap
A_{\delta}$ is finite so assume $A_{\alpha}\notin W.$ Letting $\xi
=\Delta\left(  \sigma_{\delta},\sigma_{\alpha}\right)  $ we know that
$A_{\alpha}\subseteq^{\ast}S_{\xi}^{1-\sigma_{\delta}\left(  \xi\right)  }$ so
$A_{\alpha}\cap A_{\delta}\subseteq^{\ast}\bigcup F_{\xi}$ but since $F_{\xi
}\subseteq\mathcal{D}$ we conclude that $A_{\delta}$ is almost disjoint with
$\bigcup F_{\xi}$ and then $A_{\alpha}\cap A_{\delta}$ must be finite.
\end{proof}

\qquad\ \ 

Recall that an \textsf{AD }family $\mathcal{A}$ is \emph{nowhere MAD }if for
every $X\in\mathcal{I}\left(  \mathcal{A}\right)  ^{+}$ there is $Y\in\left[
X\right]  ^{\omega}$ such that $Y$ is almost disjoint with $\mathcal{A}.$ A
key feature in the previous proof is that each $\mathcal{A}_{\delta}=\left\{
A_{\xi}\mid\xi<\delta\right\}  $ is nowhere \textsf{MAD}.

\qquad\ \ \ \ \ 

The first step to construct a $+$-Ramsey \textsf{MAD} family is to prove that
every Miller-indestructible \textsf{MAD} family has this property. If
$\mathcal{A}$ is a \textsf{MAD} family and $\mathbb{P}$ is a partial order,
then we say $\mathcal{A}$ is $\mathbb{P}$\emph{-indestructible }if
$\mathcal{A}$ is still a \textsf{MAD} family after forcing with $\mathbb{P}.$
The destructibility of \textsf{MAD} families has become a very important area
of research with many fundamental questions still open (the reader may consult
\cite{OrderingMADfamiliesalaKatetov}, \cite{ForcingwithQuotients}, or
\cite{ForcingIndestructibilityofMADFamilies} to learn more about the
indestructibility of \textsf{MAD} families and ideals). The following property
of \textsf{MAD} families plays a fundamental role in the study of destructibility:

\begin{definition}
A \textsf{MAD} family $\mathcal{A}$ is \emph{tight }if for every $\left\{
X_{n}\mid n\in\omega\right\}  \subseteq\mathcal{I}\left(  \mathcal{A}\right)
^{+}$ there is $B\in\mathcal{I}\left(  \mathcal{A}\right)  $ such that $B\cap
X_{n}$ is infinite for every $n\in\omega.$
\end{definition}

\qquad\ \ 

In \cite{OrderingMADfamiliesalaKatetov} it is proved that every tight
\textsf{MAD} family is Cohen-indestructible and that every tight \textsf{MAD}
family is $+\,$-Ramsey. We will prove that every Miller-indestructible
\textsf{MAD} family is $+$-Ramsey, this improves the previous result since
Miller-indestructibility follows from Cohen-indestructibility (see
\cite{ForcingIndestructibilityofMADFamilies}). First we need the following
lemma:\ \ \ \ \ \ \ \ \ \ \ \ \ \ \ \ \ \ \ \ \ \ \ \ \ \ \ \qquad\ 

\begin{lemma}
Let $\mathcal{A}$ be a \textsf{MAD} family and $T$ an $\mathcal{I}\left(
\mathcal{A}\right)  ^{+}$-branching tree. Then there is a subtree $S\subseteq
T$ with the following properties:

\begin{enumerate}
\item If $s\in S$ there is $A_{s}\in\mathcal{A}$ such that $suc_{S}\left(
s\right)  \in\left[  A_{s}\right]  ^{\omega}.$

\item If $s$ and $t$ are two different nodes of $S,$ then $A_{s}\neq A_{t}$
and $suc_{S}\left(  s\right)  \cap suc_{S}\left(  t\right)  =\emptyset.$
\end{enumerate}
\end{lemma}

\begin{proof}
Since $T$ is an $\mathcal{I}\left(  \mathcal{A}\right)  ^{+}$-branching tree
and $\mathcal{A}$ is \textsf{MAD}$\,$, $suc_{T}\left(  t\right)  $ infinitely
intersects many infinite elements of $\mathcal{A}$ for every $t\in T$.
Recursively, for every $t\in T$ we choose $A_{t}\in\mathcal{A}$ and $B_{t}%
\in\left[  A_{t}\cap suc_{T}\left(  t\right)  \right]  ^{\omega}$ such that
$B_{t}\cap B_{s}=\emptyset$ and $A_{s}\neq A_{t}$ whenever $t\neq s.$ We then
recursively construct $S\subseteq T$ such that if $s\in S$ then $suc_{S}%
\left(  s\right)  =B_{s}.$
\end{proof}

\qquad\ \ \ \qquad\qquad\ \ \ \qquad\ \ 

With the previous lemma we can now prove the following,

\begin{proposition}
If $\mathcal{A}$ is Miller-indestructible then $\mathcal{A}$ is $+$-Ramsey.
\end{proposition}

\begin{proof}
Let $\mathcal{A}$ be a Miller-indestructible \textsf{MAD} family and $T$ an
$\mathcal{I}\left(  \mathcal{A}\right)  ^{+}$-branching tree. Let $S$ be an
$\mathcal{I}\left(  \mathcal{A}\right)  $-branching subtree of $T$ as in the
previous lemma. We can then view $S$ as a Miller tree. Let $\dot{r}_{gen}$ be
the name of the generic real and $\dot{X}$ the name of the image of $\dot
{r}_{gen}.$

\qquad\ \ 

We will first argue that $S\Vdash``\dot{X}\notin\mathcal{I}\left(
\mathcal{A}\right)  \textquotedblright$. Assume this is not true, so there is
$S_{1}\leq S$ and $B\in\mathcal{I}\left(  \mathcal{A}\right)  $ ($B$ is an
element of $V$) such that $S_{1}\Vdash``\dot{X}\subseteq B\textquotedblright.$
In this way, if $t$ is a splitting node of $S_{1}$ then $suc_{S_{1}}\left(
t\right)  \subseteq B$ but note that if $t_{1}\neq t_{2}$ are two different
splitting nodes of $S_{2}$ then $suc_{S_{1}}\left(  t_{1}\right)  $ and
$suc_{S_{1}}\left(  t_{2}\right)  $ are two infinite sets contained in
different elements of $\mathcal{A}$, so then $B\in\mathcal{I}\left(
\mathcal{A}\right)  ^{+}$ which is a contradiction.

\qquad\ \ \ 

In this way, $\dot{X}$ is forced by $S$ to be an element of $\mathcal{I}%
\left(  \mathcal{A}\right)  ^{+}$ but since $\mathcal{A}$ is still
\textsf{MAD} after performing a forcing extension of Miller forcing, we then
conclude there are names $\{\dot{A}_{n}\mid n\in\omega\}$ for different
elements of $\mathcal{A}$ such that $S$ forces that $\dot{X}\cap\dot{A}_{n}$
is infinite. We then recursively build two sequences $\left\{  S_{n}\mid
n\in\omega\right\}  $ and $\left\{  B_{n}\mid n\in\omega\right\}  $ such that
for every $n\in\omega$ the following holds:

\begin{enumerate}
\item $S_{n}$ is a Miller tree and $B_{n}\in\mathcal{A}\mathbf{.}$

\item $S_{0}\leq S$ and if $n<m$ then $S_{m}\leq S_{n}.$

\item $S_{n}\Vdash``\dot{A}_{n}=B_{n}\textquotedblright$ (it then follows that
$B_{n}\neq B_{m}$ if $n\neq m$).

\item If $i\leq n$ then $stem\left(  S_{n}\right)  \cap B_{i}$ has size at
least $n.$
\end{enumerate}

\qquad\ \ \ \ \ \ \qquad\ \ 

We then define $r=%
{\textstyle\bigcup\limits_{n\in\omega}}
stem\left(  S_{n}\right)  $ then clearly $r\in\left[  S\right]  $ and
$im\left(  r\right)  \in\mathcal{I}\left(  \mathcal{A}\right)  ^{+}%
.$\ \ \qquad\ \ \qquad\ \ 
\end{proof}

\qquad\ \ \ \ \ 

The converse of the previous result is not true in general, this will be shown
in the corollary 27. It is known that every \textsf{MAD }family of size less
than $\mathfrak{d}$ is Miller indestructible (see
\cite{ForcingIndestructibilityofMADFamilies}). We can then conclude the
following unpublished result of Michael Hru\v{s}\'{a}k, which he proved by
completely different means.\qquad\ \ \qquad\ \ 

\begin{corollary}
[Hru\v{s}\'{a}k]Every \textsf{MAD} family of size less than $\mathfrak{d}$ is
$+$-Ramsey. In particular, if $\mathfrak{a<d}$ then there is a $+$-Ramsey
\textsf{MAD} family.
\end{corollary}

\qquad\ \ \ 

\section*{The construction of a +-Ramsey MAD family}

In this chapter we will construct a $+$-Ramsey \textsf{MAD} family without any
extra hypothesis beyond \textsf{ZFC}$\mathsf{.}$ We will use the construction
of Shelah of a completely separable \textsf{MAD} family, however, the previous
result will help us avoid the need of a \textsf{PCF }hypothesis for our
construction. From now on, we will always assume that all Miller trees are
formed by increasing sequences. If $p$ is a Miller tree, we denote
$Split\left(  p\right)  $ the set of all splitting nodes of $p.$

\begin{definition}
Let $p$ be a Miller tree. Given $f\in\left[  p\right]  $ we define $Sp\left(
p,f\right)  =\left\{  f\left(  n\right)  \mid f\upharpoonright n\in
Split\left(  p\right)  \right\}  $ and $\left[  p\right]  _{split}=\left\{
Sp\left(  p,f\right)  \mid f\in\left[  p\right]  \right\}  .$
\end{definition}

\qquad\ \ \ \ 

We will need the following definitions,\qquad\ \ \ \ \qquad\ \ \ \qquad

\begin{definition}
Let $p$ be a Miller tree and $H:Split\left(  p\right)  \longrightarrow\omega.$
We then define:

\begin{enumerate}
\item $Catch_{\exists}\left(  H\right)  $ is the set \qquad

$\left\{  Sp\left(  f,p\right)  \mid f\in\left[  p\right]  \wedge
\exists^{\infty}n\left(  f\upharpoonright n\in Split\left(  p\right)  \wedge
f\left(  n\right)  <H\left(  f\upharpoonright n\right)  \right)  \right\}  .$

\item $Catch_{\forall}\left(  H\right)  $ is the set \qquad

$\left\{  Sp\left(  f,p\right)  \mid f\in\left[  p\right]  \wedge
\forall^{\infty}n\left(  f\upharpoonright n\in Split\left(  p\right)  \wedge
f\left(  n\right)  <H\left(  f\upharpoonright n\right)  \right)  \right\}  .$

\item Define $\mathcal{K}\left(  p\right)  $ as the collection of all
$A\subseteq\left[  p\right]  _{split}$ for which there is $G:Split\left(
p\right)  \longrightarrow\omega$ such that $A\subseteq Catch_{\exists}\left(
G\right)  .$
\end{enumerate}
\end{definition}

\qquad\ \ \ \qquad\ \ 

Note that if $\mathcal{B}=\left\{  f_{\alpha}\mid\alpha<\mathfrak{b}\right\}
\subseteq\omega^{\omega}$ is an unbounded family of increasing functions then
for every infinite partial function $g\subseteq\omega\times\omega$ there is
$\alpha<\mathfrak{b}$ such that the set $\left\{  n\in dom\left(  g\right)
\mid g\left(  n\right)  <f_{\alpha}\left(  n\right)  \right\}  $ is infinite.
With this simple observation we can prove the following lemma,

\begin{lemma}
$\mathcal{K}\left(  p\right)  $ is a $\sigma$-ideal in $\left[  p\right]
_{split}$ that contains all singletons and $\mathfrak{b=}$ \textsf{add}%
$\left(  \mathcal{K}\left(  p\right)  \right)  =$ \textsf{cov}$\left(
\mathcal{K}\left(  p\right)  \right)  .$
\end{lemma}

\begin{proof}
In order to prove that $\mathfrak{b}$ $\mathfrak{\leq}$ \textsf{add}$\left(
\mathcal{K}\left(  p\right)  \right)  $ it is enough to show that if
$\kappa<\mathfrak{b}$ and $\left\{  H_{\alpha}\mid\alpha<\kappa\right\}
\subseteq\omega^{Split\left(  p\right)  }$ then $%
{\textstyle\bigcup\limits_{\alpha<\kappa}}
Catch_{\exists}\left(  H_{\alpha}\right)  \in\mathcal{K}\left(  p\right)  .$
Since $\kappa$ is smaller than $\mathfrak{b}$, we can find $H:Split\left(
p\right)  \longrightarrow\omega$ such that if $\alpha<\kappa$ then $H_{\alpha
}\left(  s\right)  <H\left(  s\right)  $ for almost all $s\in Split\left(
p\right)  .$ Clearly $%
{\textstyle\bigcup\limits_{\alpha<\kappa}}
Catch_{\exists}\left(  H_{\alpha}\right)  \subseteq Catch_{\exists}\left(
H\right)  .$

\qquad\ \ \qquad\qquad\ \ \ \ \ 

Now we must prove that \textsf{cov}$\left(  \mathcal{K}\left(  p\right)
\right)  \leq\mathfrak{b}.$ Let $Split\left(  p\right)  =\left\{  s_{n}\mid
n<\omega\right\}  $ and $\mathcal{B}=\left\{  f_{\alpha}\mid\alpha
<\mathfrak{b}\right\}  \subseteq\omega^{\omega}$ be an unbounded family of
increasing functions. Given $\alpha<\mathfrak{b}$ define $H_{\alpha
}:Split\left(  p\right)  \longrightarrow\omega$ where $H_{\alpha}\left(
s_{n}\right)  =$ $f_{\alpha}\left(  n\right)  .$ We will show that $\left\{
Catch_{\exists}\left(  H_{\alpha}\right)  \mid\alpha<\mathfrak{b}\right\}  $
covers $\left[  p\right]  _{split}.$ Letting $f\in\left[  p\right]  $ define
$A=\left\{  n\mid s_{n}\sqsubseteq f\right\}  $ and construct the function
$g:A\longrightarrow\omega$ where $g\left(  n\right)  =f\left(  \left\vert
s_{n}\right\vert \right)  +1$ for every $n\in A.$ By the previous remark,
there is $\alpha<\mathfrak{b}$ such that $f_{\alpha}\upharpoonright A$ is not
dominated by $g\upharpoonright A.$ It is then clear that $S_{p}\left(
p,f\right)  \in Catch_{\exists}\left(  H_{\alpha}\right)  .$ \qquad
\ \ \qquad\ \ \ \qquad\ \ \ \ 
\end{proof}

\qquad\ \ \ \ \ \ \ \ \ \ \ \ \ \ \ \ \ \ \ \ \ \ \ \ \ \ \ \ \ \ \ \ \ \ \ \ \ \ \ \ \ \ \ \ \ \ \ \ \ \ \ 

Letting $p$ be a Miller tree and $S\in\left[  \omega\right]  ^{\omega},$ we
define the game $\mathcal{G}\left(  p,S\right)  $ as follows:

\qquad\ \ 

\begin{center}%
\begin{tabular}
[c]{|l|l|l|l|l|l|}\hline
$\mathsf{I}$ & $s_{0}$ &  & $s_{1}$ &  & $\cdots$\\\hline
$\mathsf{II}$ &  & $r_{0}$ &  & $r_{1}$ & \\\hline
\end{tabular}

\end{center}

\qquad\ \ \qquad\ 

\begin{enumerate}
\item Each $s_{i}$ is a splitting node of $p.$

\item $r_{i}\in\omega.$

\item $s_{i+1}$ extends $s_{i}.$

\item $s_{i+1}\left(  \left\vert s_{i}\right\vert \right)  \in S$ and is
bigger than $r_{i}.$
\end{enumerate}

\qquad\ \ \ \ \ 

Player $\mathsf{I}$ wins the game if she can continue playing for infinitely
many rounds. Given $S\in\left[  \omega\right]  ^{\omega}$ we denote by
$Hit\left(  S\right)  $ as the set of all subsets of $\omega$ that have
infinite intersection with $S.$

\begin{lemma}
Letting $p$ be a Miller tree and $S\in\left[  \omega\right]  ^{\omega}$, for
the game $\mathcal{G}\left(  p,S\right)  $ we have the following:

\begin{enumerate}
\item Player $I$ has a winning strategy if and only if there is $q\leq p$ such
that $\left[  q\right]  _{split}\subseteq\left[  S\right]  ^{\omega}.$

\item Player $II$ has a winning strategy if and only if there is
$H:Split\left(  p\right)  \longrightarrow\omega$ such that if $f\in\left[
p\right]  $ then the set $\left\{  f\upharpoonright n\in Split\left(
p\right)  \mid f\left(  n\right)  \in S\right\}  $ is almost contained in
$\left\{  f\upharpoonright n\in Split\left(  p\right)  \mid f\left(  n\right)
<H\left(  f\upharpoonright n\right)  \right\}  $ (in particular $\left[
p\right]  _{split}\cap Hit\left(  S\right)  \in\mathcal{K}\left(  p\right)  $).
\end{enumerate}
\end{lemma}

\begin{proof}
The first equivalence is easy so we leave it for the reader. Now assume there
is a winning strategy $\pi$ for $II.$ We define $H:Split\left(  p\right)
\longrightarrow\omega$ such that if $s\in Split\left(  p\right)  $ then
$\pi\left(  \overline{x}\right)  <H\left(  s\right)  $ where $\overline{x}$ is
any partial play in which player $I$ has build $s$ and $II$ has played
according to $\pi$ (note there are only finitely many of those $\overline{x}$
so we can define $H\left(  s\right)  $). We want to prove that if $f\in\left[
p\right]  $ then $\left\{  f\upharpoonright n\in Split\left(  p\right)  \mid
f\left(  n\right)  \in S\right\}  $ is almost contained in the set $\left\{
f\upharpoonright n\in Split\left(  p\right)  \mid f\left(  n\right)  <H\left(
f\upharpoonright n\right)  \right\}  .$ Assume this is not the case. $\ $Let
$B$ be the set of all $n\in\omega$ such that $f\upharpoonright n\in
Split\left(  p\right)  $ with $f\left(  n\right)  \in S$ but $H\left(
f\upharpoonright n\right)  \leq f\left(  n\right)  .$ By our hypothesis $B$ is
infinite and then we enumerate it as $B=\left\{  b_{n}\mid n\in\omega\right\}
$ in increasing order. Consider the run of the game where $\mathsf{I}$ plays
$f\upharpoonright b_{n}$ at the $n$-th stage. This is possible since $f\left(
b_{n}\right)  \in S$ and $H\left(  f\upharpoonright b_{n}\right)  \leq
f\left(  b_{n}\right)  $ so $\mathsf{I}$ will win the game, which is a
contradiction. The other implication is easy.
\end{proof}

\qquad\ \ 

Since $\mathcal{G}\left(  p,S\right)  $ is an open game for $\mathsf{II}$ by
the Gale-Stewart theorem (see \cite{Kechris}) it is determined, so we conclude
the following dichotomy:\ \ \ \ \ \ \ \ \ \ \ \ \ \ 

\begin{corollary}
If $p$ is a Miller tree and $S\in\left[  \omega\right]  ^{\omega}$ then one
and only one of the following holds:

\begin{enumerate}
\item There is $q\leq p$ such that $\left[  q\right]  _{split}\subseteq\left[
S\right]  ^{\omega}.$

\item There is $H:Split\left(  p\right)  \longrightarrow\omega$ such that if
$f\in\left[  p\right]  $ then the set defined as $\left\{  f\upharpoonright
n\in Split\left(  p\right)  \mid f\left(  n\right)  \in S\right\}  $ is almost
contained in the following set: $\left\{  f\upharpoonright n\in Split\left(
p\right)  \mid f\left(  n\right)  <H\left(  f\upharpoonright n\right)
\right\}  $ (and $\left[  p\right]  _{split}\cap Hit\left(  S\right)
\in\mathcal{K}\left(  p\right)  $).
\end{enumerate}
\end{corollary}

\qquad\ \ \ 

In particular, for every Miller tree $p$ and $S\in\left[  \omega\right]
^{\omega}$ there is $q\leq p$ such that either $\left[  q\right]
_{split}\subseteq\left[  S\right]  ^{\omega}$\ or\ $\left[  q\right]
_{split}\subseteq\left[  \omega\setminus S\right]  ^{\omega}$ \ (although this
fact can be proved easier without the game).\ \ \ \ \ \ \ \ 

\begin{definition}
Let $p$ be a Miller tree and $S\in\left[  \omega\right]  ^{\omega}.$ We say
$S$ \emph{tree-splits }$p$ if there are Miller trees $q_{0},q_{1}\leq p$ such
that $\left[  q_{0}\right]  _{split}\subseteq\left[  S\right]  ^{\omega}$ and
$\left[  q_{1}\right]  _{split}\subseteq\left[  \omega\setminus S\right]
^{\omega}.$ $\mathcal{S}$ is a \emph{Miller tree-splitting family }if every
Miller tree is tree-split by some element of $\mathcal{S}.$\ 
\end{definition}

\qquad\ \ \ 

It is easy to see that every Miller-tree splitting family is a splitting
family and it is also easy to see that every $\omega$-splitting family is a
Miller-tree splitting family. We will now prove there is a Miller-tree
splitting family of size $\mathfrak{s}.$ I want to thank the referee for
supplying the following argument which is simpler than the original
one:\qquad\ \ \ \qquad\ \ \ \ \ \ \ \ 

\begin{proposition}
The smallest size of a Miller-tree splitting family is $\mathfrak{s}.$
\end{proposition}

\begin{proof}
We will construct a Miller-tree splitting family of size $\mathfrak{s.}$ In
case $\mathfrak{b\leq s}$ there is an $\omega$-splitting family of size
$\mathfrak{s}$ (see proposition 5 and remark after definition 6) and this is a
Miller tree-splitting family as remarked above.

\ \qquad\ \ \ \ 

Now assume $\mathfrak{s}<\mathfrak{b}.$ We will show that any splitting family
of size $\mathfrak{s}$ is a Miller tree-splitting family. We argue by
contradiction, let $\mathcal{S}=\left\{  S_{\alpha}\mid\alpha<\mathfrak{s}%
\right\}  $ be a splitting family which does not tree-splits the Miller tree
$p.$ In this way, for every $\alpha<\mathfrak{s}$ there is $i\left(
\alpha\right)  <2$ such that there is no $q\leq p$ for which $\left[
q\right]  _{split}\subseteq\lbrack S_{\alpha}^{i\left(  \alpha\right)
}]^{\omega}.$ By corollary 18, there is $H_{\alpha}:Split\left(  p\right)
\longrightarrow\longrightarrow\omega$\qquad\ such that for every $f\in\left[
p\right]  $ the following holds:

\qquad\ \qquad\ \qquad\ \qquad\ \qquad\ \qquad\ \qquad\ \ \qquad%

\begin{tabular}
[c]{l}%
$\left\{  f\upharpoonright n\in Split\left(  p\right)  \mid f\left(  n\right)
\in S\right\}  \subseteq^{\ast}\left\{  f\upharpoonright n\in Split\left(
p\right)  \mid f\left(  n\right)  <H\left(  f\upharpoonright n\right)
\right\}  $%
\end{tabular}

\qquad\ \ \ \ \ 

or equivalently:

\qquad\ \qquad\ \qquad\ \qquad\ \qquad\ \qquad\ \qquad\ \ \qquad%

\begin{tabular}
[c]{l}%
$\forall^{\infty}n\left(  f\upharpoonright n\in Split\left(  p\right)  \wedge
H_{\alpha}\left(  f\upharpoonright n\right)  \leq f\left(  n\right)
\longrightarrow f\left(  n\right)  \in S_{\alpha}^{1-i\left(  \alpha\right)
}\right)  $%
\end{tabular}

\qquad\ \ \qquad\ \ \qquad\ \ 

Since $\mathfrak{s<b}$ there exists $H:Split\left(  p\right)  \longrightarrow
\omega$ dominating each $H_{\alpha}.$ Take $f\in\left[  p\right]  $ such that
for every $n\in\omega$ if $f\upharpoonright n\in Split\left(  p\right)  $ then
$f\left(  n\right)  >H\left(  f\upharpoonright n\right)  .$ Then,

\qquad\ \qquad\ \qquad\ \qquad\ \qquad\ \qquad\ \qquad\ \ \qquad%

\begin{tabular}
[c]{l}%
$\forall\alpha<\mathfrak{s}\forall^{\infty}n\left(  f\upharpoonright n\in
Split\left(  p\right)  \longrightarrow f\left(  n\right)  \in S_{\alpha
}^{1-i\left(  \alpha\right)  }\right)  $%
\end{tabular}

\qquad\ \ \qquad\ \ \ \ \ \ 

Let $X=\left\{  f\left(  n\right)  \mid f\upharpoonright n\in Split\left(
p\right)  \right\}  $, note that $X\subseteq^{\ast}S_{\alpha}^{1-i\left(
\alpha\right)  }$ for every $\alpha<\mathfrak{s}.$ But this contradicts that
$\mathcal{S}$ was a splitting family.\ \ 
\end{proof}

\qquad\ \qquad\ \qquad\ \ \ \ \ 

The following lemma is probably well known:

\begin{lemma}
Assume $\kappa<\mathfrak{d}$ and for every $\alpha<\kappa$ let $\mathcal{F}%
_{\alpha}\subseteq\left[  \omega\right]  ^{<\omega}$ be an infinite set of
disjoint finite subsets of $\omega$ and $g_{\alpha}:\bigcup\mathcal{F}%
_{\alpha}\longrightarrow\omega.$ Then there is $f:\omega\longrightarrow\omega$
such that for every $\alpha<\kappa$ there are infinitely many $X\in
\mathcal{F}_{\alpha}$ such that $g_{\alpha}\upharpoonright X<f\upharpoonright
X.$
\end{lemma}

\begin{proof}
Given $\alpha<\kappa$ find an interval partition $\mathcal{P}_{\alpha
}=\left\{  P_{\alpha}\left(  n\right)  \mid n\in\omega\right\}  $ such that
for every $n\in\omega$ there is $X\in\mathcal{F}_{\alpha}$ such that
$X\subseteq P_{\alpha}\left(  n\right)  $ (this is possible since
$\mathcal{F}_{\alpha}$ is infinite and its elements are pairwise disjoint).
Then define the function $\overline{g}_{\alpha}:\omega\longrightarrow\omega$
such that $\overline{g}_{\alpha}\upharpoonright P_{\alpha}\left(  n\right)  $
is the constant function $max\left\{  g_{\alpha}\left[  P_{\alpha}\left(
n+1\right)  \right]  \right\}  .$ Since $\kappa$ is smaller than
$\mathfrak{d},$ we can then find an increasing function $f:\omega
\longrightarrow\omega$ that is not dominated by any of the $\overline
{g}_{\alpha}.$ It is easy to prove that $f$ has the desired property.
\end{proof}

\qquad\ \ \ \ 

Now we can prove the following lemma that will be useful:

\begin{lemma}
Let $q$ be a Miller tree and $\kappa<\mathfrak{d}.$ If $\left\{  H_{\alpha
}\mid\alpha<\kappa\right\}  \subseteq\omega^{Split\left(  q\right)  }$ then
there is $r\leq q$ such that $Split\left(  r\right)  =Split\left(  q\right)
\cap r$ and $\left[  r\right]  _{split}\cap%
{\displaystyle\bigcup\limits_{\alpha<\kappa}}
Catch_{\forall}\left(  H_{\alpha}\right)  =\emptyset$.
\end{lemma}

\begin{proof}
We will first prove there is $G:Split\left(  q\right)  \longrightarrow\omega$
such that $%
{\displaystyle\bigcup\limits_{\alpha<\kappa}}
Catch_{\forall}\left(  H_{\alpha}\right)  $ is a subset of $Catch_{\exists
}\left(  G\right)  .$ Given $t\in Split\left(  q\right)  $ let $T\left(
t,\alpha\right)  $ the subtree of $q$ such that if $f\in\left[  T\left(
t,\alpha\right)  \right]  $ then $t\sqsubseteq f$ and if $t\sqsubseteq
f\upharpoonright n$ and $f\upharpoonright n\in Split\left(  q\right)  $ then
$f\left(  n\right)  \in H_{\alpha}\left(  f\upharpoonright n\right)  .$
Clearly $T\left(  t,\alpha\right)  $ is a finitely branching subtree of $q.$
Then define $\mathcal{F}\left(  t,\alpha\right)  =\left\{  Split_{n}\left(
q\right)  \cap T\left(  t,\alpha\right)  \mid n<\omega\right\}  $ which is an
infinite collection of pairwise disjoint finite sets and let $g_{\left(
t,\alpha\right)  }:\bigcup\mathcal{F}\left(  t,\alpha\right)  \longrightarrow
\omega$ given by $g_{\left(  t,\alpha\right)  }\left(  s\right)  =H_{\alpha
}\left(  s\right)  .$ Since $\kappa<\mathfrak{d}$ by the previous lemma, we
can find $G:Split\left(  q\right)  \longrightarrow\omega$ such that if
$\alpha<\kappa$ and $t\in Split\left(  q\right)  $ then there are infinitely
many $Y\in\mathcal{F}\left(  t,\alpha\right)  $ such that $g_{\left(
t,\alpha\right)  }\upharpoonright Y<G\upharpoonright Y.$ We will now prove
that $%
{\displaystyle\bigcup\limits_{\alpha<\kappa}}
Catch_{\forall}\left(  H_{\alpha}\right)  \subseteq Catch_{\exists}\left(
G\right)  .$ Let $\alpha<\kappa$ and $f\in Catch_{\forall}\left(  H_{\alpha
}\right)  .$ Find $t\in Split\left(  q\right)  $ such that $t\sqsubseteq f$
and if $t\sqsubseteq f\upharpoonright m$ and $f\upharpoonright m\in
Split\left(  q\right)  $ then $f\left(  m\right)  \in H_{\alpha}\left(
f\upharpoonright m\right)  .$ Note that $f$ is a branch through $T\left(
t,\alpha\right)  .$ Let $Y\in\mathcal{F}\left(  t,\alpha\right)  $ such that
$g_{\left(  t,\alpha\right)  }\upharpoonright Y<G\upharpoonright Y$ and since
$f\in\left[  T\left(  t,\alpha\right)  \right]  $, there is $n\in\omega$ such
that $f\upharpoonright n\in Y$ so $f\left(  n\right)  <H_{\alpha}\left(
f\upharpoonright n\right)  <G\left(  f\upharpoonright n\right)  .$

\qquad\ \ \ \ 

Define $r\leq q$ such that $Split\left(  r\right)  =Split\left(  q\right)
\cap r$ and $suc_{r}\left(  s\right)  =suc_{q}\left(  s\right)  \setminus
G\left(  s\right)  .$ Clearly $\left[  r\right]  _{split}$ is disjoint from
$Catch_{\exists}\left(  G\right)  .$
\end{proof}

\qquad\ \ \ 

We can then finally prove our main theorem.

\begin{theorem}
There is a $+$-Ramsey \textsf{MAD} family.
\end{theorem}

\begin{proof}
If $\mathfrak{a<s},$ then $\mathfrak{a}$ is smaller than $\mathfrak{d}$ so
then there is a $+$-Ramsey \textsf{MAD} family (in fact, there is a
Miller-indestructible \textsf{MAD} family, see corollary 13). So we assume
$\mathfrak{s}\leq\mathfrak{a}$ for the rest of the proof. Fix an $\left(
\omega,\omega\right)  $-splitting family $\mathcal{S}=\left\{  S_{\alpha}%
\mid\alpha<\mathfrak{s}\right\}  $ that is also a Miller-tree splitting
family. Let $\left\{  L,R\right\}  $ be a partition of the limit ordinals
smaller than $\mathfrak{c}$ such that both $L$ and $R$ have size continuum.
Enumerate by $\left\{  X_{\alpha}\mid\alpha\in L\right\}  $ all infinite
subsets of $\omega$ and by $\left\{  T_{\alpha}\mid\alpha\in R\right\}  $ all
subtrees of $\omega^{<\omega}.$ We will recursively construct $\mathcal{A}%
=\left\{  A_{\xi}\mid\xi<\mathfrak{c}\right\}  $ and $\left\{  \sigma_{\xi
}\mid\xi<\mathfrak{c}\right\}  $ as follows:

\begin{enumerate}
\item $\mathcal{A}$ is an \textsf{AD} family and $\sigma_{\alpha}%
\in2^{<\mathfrak{s}}$ for every $\alpha<\mathfrak{c}.$

\item If $\sigma_{\alpha}\in2^{\beta}$ and $\xi<\beta$ then $A_{\alpha
}\subseteq^{\ast}S_{\xi}^{\sigma_{\alpha}\left(  \xi\right)  }.$

\item If $\alpha\neq\beta$ then $\sigma_{\alpha}\neq\sigma_{\beta}.$

\item If $\delta\in L$ and $X_{\delta}\in\mathcal{I}\left(  \mathcal{A}%
_{\delta}\right)  ^{+}$ then $A_{\delta+n}\subseteq X_{\delta}$ for every
$n\in\omega$ (where $\mathcal{A}_{\delta}=\left\{  A_{\xi}\mid\xi
<\delta\right\}  $).

\item If $\delta\in R$ and $T_{\delta}$ is an $\mathcal{I}\left(
\mathcal{A}_{\delta}\right)  ^{+}$-branching tree then there is $f\in\left[
T_{\delta}\right]  $ such that $A_{\delta+n}\subseteq im\left(  f\right)  $
for every $n\in\omega.$
\end{enumerate}

\qquad\ \ 

It is clear that if we manage to perform the construction then $\mathcal{A}$
will be a $\mathfrak{+}$-Ramsey \textsf{MAD} family (and it will be completely
separable too). Let $\delta$ be a limit ordinal and assume we have constructed
$A_{\xi}$ for every $\xi<\delta.$ In case $\delta\in L$ we just proceed as in
the case of the completely separable \textsf{MAD} family, so assume $\delta\in
R.$ Since $\mathcal{A}_{\delta}=\left\{  A_{\xi}\mid\xi<\delta\right\}  $ is
nowhere-\textsf{MAD }(recall that $\mathcal{A}_{\delta}$ is
nowhere-\textsf{MAD }by the proof of theorem 9) we can find $p$ an
$\mathcal{A}_{\delta}^{\perp}$-branching subtree of $T_{\delta}$ (recall that
$\mathcal{A}_{\delta}^{\perp}$ is the set of all infinite sets that are almost
disjoint with every element of $\mathcal{A}_{\delta}^{\perp}$).

\qquad\ \ 

First note that since $\mathcal{S}$ is a Miller-tree splitting family, for
every Miller tree $q$ there is $\alpha<\mathfrak{s}$ and $\tau_{q}\in
2^{\alpha}$ such that:

\begin{enumerate}
\item $S_{\alpha}$ tree-splits $q.$

\item If $\xi<\alpha$ then there is no $q^{\prime}\leq q$ such that $\left[
q^{\prime}\right]  _{split}\subseteq\lbrack S_{\xi}^{1-\tau_{q}\left(
\xi\right)  }]^{\omega}.$
\end{enumerate}

\qquad\ \ \ 

Note that if $q^{\prime}\leq q$ then $\tau_{q^{\prime}}$ extends $\tau_{q}.$
If $q\leq p$ and $\tau_{q}\in2^{\alpha}$ we fix the following items:

\begin{enumerate}
\item $W_{0}\left(  q\right)  =\left\{  \xi<\alpha\mid\exists\beta
<\delta\left(  \sigma_{\beta}=\tau_{q}\upharpoonright\xi\right)  \right\}  $
and\qquad\ \ \ \ \ \ 

$W_{1}\left(  q\right)  =\left\{  \xi<\alpha\mid\exists\beta<\delta\left(
\Delta\left(  \sigma_{\beta},\tau_{q}\right)  =\xi\right)  \right\}
.$\footnote{By $\Delta\left(  \sigma,\tau\right)  $ we denote the smallest
$\beta$ such that $\sigma\left(  \beta\right)  \neq\tau\left(  \beta\right)  $
in case $\sigma$ and $\tau$ are incomparable or if $\tau\upharpoonright
\beta=\sigma.$}

\item Let $\xi\in W_{0}\left(  q\right)  $ we then find $\beta$ such that
$\sigma_{\beta}=\tau_{q}\upharpoonright\xi$ and define $G_{q,\xi}:Split\left(
q\right)  \longrightarrow\omega$ such that if $s\in Split\left(  q\right)  $
then $A_{\beta}\cap suc_{q}\left(  s\right)  \subseteq G_{q,\xi}\left(
s\right)  $ (this is possible since $q$ is $\mathcal{A}_{\delta}^{\perp}$-branching).

\item Given $\xi$ $\in W_{1}\left(  q\right)  $ we know there is no
$q^{\prime}\leq q$ such that $\left[  q^{\prime}\right]  _{split}%
\subseteq\lbrack S_{\xi}^{1-\tau_{q}\left(  \xi\right)  }]^{\omega}.$ We know
that there is $H_{q,\xi}:Split\left(  q\right)  \longrightarrow\omega$ such
that if $f\in\left[  q\right]  $, the set defined as $\{f\upharpoonright n\in
Split\left(  q\right)  \mid f\left(  n\right)  \in S_{\xi}^{1-\tau_{q}\left(
\xi\right)  }\}$ is almost contained in the set $\left\{  f\upharpoonright
n\in Split\left(  q\right)  \mid f\left(  n\right)  <H_{q,\xi}\left(
f\upharpoonright n\right)  \right\}  .$

\item If $U\in\left[  W_{0}\left(  q\right)  \right]  ^{<\omega}$ and
$V\in\left[  W_{1}\left(  q\right)  \right]  ^{<\omega}$ choose any \qquad\ \ \ \ 

$J_{q,U,V}:Split\left(  q\right)  \longrightarrow\omega$ such that if $s\in
Split\left(  q\right)  $ then $J_{q,U,V}\left(  s\right)  >max\left\{
G_{q,\xi}\left(  s\right)  \mid\xi\in U\right\}  ,$ $max\left\{  H_{q,\xi
}\left(  s\right)  \mid\xi\in V\right\}  .$

\item $\mathcal{A}\left(  q\right)  =\left\{  A_{\xi}\in\mathcal{A}_{\delta
}\mid\tau_{q}\nsubseteq\sigma_{\xi}\right\}  .$
\end{enumerate}

\qquad\ \ \ \qquad\ \ \ \ \ \ \ \ \ \qquad\ \ \ \ \ 

Note that if $\xi\in W_{0}\left(  q\right)  $ then there is a unique
$\beta<\delta$ such that $\sigma_{\beta}=\tau_{q}\upharpoonright\xi$ (although
the analogous remark is not true for the elements of $W_{1}\left(  q\right)
$). The following claim will play a fundamental role in the proof:

\begin{claim}
If $q\leq p$ then there is $r\leq q$ such that $\left[  r\right]
_{split}\subseteq\mathcal{I}\left(  \mathcal{A}\left(  q\right)  \right)
^{+}$.
\end{claim}

\qquad\ \ \ \ 

Let $\alpha<\mathfrak{s}$ such that $\tau_{q}\in2^{\alpha}.$ Since
$\mathfrak{s\leq d}$, we know there is $r\leq q$ such that $\left[  r\right]
_{split}$ is disjoint from $\bigcup\left\{  Catch_{\forall}\left(
J_{q,U,V}\right)  \mid U\in\left[  W_{0}\left(  q\right)  \right]  ^{<\omega
},V\in\left[  W_{1}\left(  q\right)  \right]  ^{<\omega}\right\}  $ and
$Split\left(  r\right)  =Split\left(  q\right)  \cap r.$ We will now prove
$\left[  r\right]  _{split}\subseteq\mathcal{I}\left(  \mathcal{A}\left(
q\right)  \right)  ^{+}$ but assume this is not the case. Therefore, there is
$f\in\left[  r\right]  $ and $F\in\left[  \mathcal{A}\left(  q\right)
\right]  ^{<\omega}$ such that $X=Sp\left(  r,f\right)  \subseteq^{\ast
}\bigcup F.$ Let $F=F_{1}\cup F_{2}$ and $U\in\left[  W_{0}\left(  q\right)
\right]  ^{<\omega}$, $V\in\left[  W_{1}\left(  q\right)  \right]  ^{<\omega}$
such that for every $A_{\beta}\in F_{1}$ there is $\xi_{\beta}\in U$ such that
$\sigma_{\beta}=\tau_{q}\upharpoonright\xi_{\beta}$ and for every $A_{\gamma
}\in F_{2}$ there is $\eta_{\gamma}\in V$ such that $\Delta\left(  \tau
_{q},\sigma_{\gamma}\right)  =\eta_{\gamma}.$ Let $D\subseteq\left\{  n\mid
f\upharpoonright n\in Split\left(  r\right)  \right\}  $ be the (infinite) set
of all $n<\omega$ such that the following holds:

\qquad\ \ 

\begin{enumerate}
\item $f\upharpoonright n\in Split\left(  r\right)  $ and $f\left(  n\right)
\in\bigcup F.$

\item If $\eta_{\gamma}\in V$ then $A_{\gamma}\setminus n\subseteq
S_{\eta_{\gamma}}^{1-\tau_{q}\left(  \eta_{\gamma}\right)  }.$

\item $f\left(  n\right)  >J_{q,U,V}\left(  f\upharpoonright n\right)  .$

\item If $\eta\in V$ and $f\left(  n\right)  \in S_{\eta}^{1-\tau_{q}\left(
\eta\right)  }$ then $f\left(  n\right)  <H_{q,\eta}\left(  f\upharpoonright
n\right)  <J_{q,U,V}\left(  f\upharpoonright n\right)  $ (recall that
$\{f\upharpoonright m\in Split\left(  q\right)  \mid f\left(  m\right)  \in
S_{\eta}^{1-\tau_{q}\left(  \eta\right)  }\}$ is almost contained in $\left\{
f\upharpoonright m\in Split\left(  q\right)  \mid f\left(  m\right)
<H_{q,\eta}\left(  f\upharpoonright m\right)  \right\}  $).
\end{enumerate}

\qquad\ \ 

We first claim that if $n\in D,\xi_{\beta}\in U$ and $\eta_{\gamma}\in$ $V$
then $f\left(  n\right)  \notin A_{\beta}\cup S_{\eta_{\gamma}}^{1-\tau
_{q}\left(  \eta_{\gamma}\right)  }.$ On one hand, since $A_{\beta}\cap
suc_{q}\left(  f\upharpoonright n\right)  \subseteq G_{q,\xi_{\beta}}\left(
f\upharpoonright n\right)  <J_{q,U,V}\left(  f\upharpoonright n\right)  $ and
$f\left(  n\right)  >J_{q,U,V}\left(  f\upharpoonright n\right)  $ then
$f\left(  n\right)  \notin A_{\beta}.$ On the other hand, if it was the case
that $f\left(  n\right)  \in S_{\eta_{\gamma}}^{1-\tau_{q}\left(  \eta
_{\gamma}\right)  }$ so $f\left(  n\right)  <H_{q,\eta}\left(
f\upharpoonright n\right)  <J_{q,U,V}\left(  f\upharpoonright n\right)  $ but
we already know that $f\left(  n\right)  >J_{q,U,V}\left(  f\upharpoonright
n\right)  .$ Since $n\leq f\left(  n\right)  $ (recall every branch through
$p$ is increasing) $f\left(  n\right)  \notin A_{\gamma}$ for every
$\eta_{\gamma}\in V$ because $A_{\gamma}\setminus n\subseteq S_{\eta_{\gamma}%
}^{1-\tau_{q}\left(  \eta_{\gamma}\right)  }.$ This implies $f\left(
n\right)  \notin\bigcup F$ which is a contradiction and finishes the proof of
the claim.

\qquad\ \ 

Back to the proof of the theorem, we recursively build a tree of Miller trees
$\left\{  p\left(  s\right)  \mid s\in2^{<\omega}\right\}  $ with the
following properties:

\qquad\ \ \ 

\begin{enumerate}
\item $p\left(  \emptyset\right)  =p.$

\item $p\left(  s^{\frown}i\right)  \leq p\left(  s\right)  $ and the stem of
$p\left(  s^{\frown}i\right)  $ has length at least $\left\vert s\right\vert
.$

\item $\tau_{p\left(  s^{\frown}0\right)  }$ and $\tau_{p\left(  s^{\frown
}1\right)  }$ are incompatible.

\item $\left[  p\left(  s^{\frown}i\right)  \right]  _{split}\subseteq
\mathcal{I}\left(  \mathcal{A}\left(  p\left(  s\right)  \right)  \right)
^{+}.$
\end{enumerate}

\qquad\ \ \ 

This is easy to do with the aid of the previous claim. For every
$g\in2^{\omega}$ let $\tau_{g}=\bigcup\tau_{p\left(  g\upharpoonright
m\right)  }.$ Note that if $g_{1}\neq g_{2}$ then $\tau_{g_{1}}$ and
$\tau_{g_{2}}$ are two incompatible nodes of $2^{<\mathfrak{s}}.$ Since
$\mathcal{A}_{\delta}$ has size less than the continuum, there is
$g\in2^{\omega}$ such that there is no $\beta<\delta$ such that $\sigma
_{\beta}$ extends $\tau_{g}$ and then $\mathcal{A}_{\delta}=\bigcup
\limits_{m\in\omega}\mathcal{A}\left(  p\left(  g\upharpoonright m\right)
\right)  .$ Let $f$ be the only element of $\bigcap\limits_{m\in\omega}\left[
p\left(  g\upharpoonright m\right)  \right]  .$ Obviously, $f$ is a branch
through $p$ and we claim that $im\left(  f\right)  \in\mathcal{I}\left(
\mathcal{A}_{\delta}\right)  ^{+}.$ This is easy since if $A_{\xi_{1}%
},...,A_{\xi_{n}}\in\mathcal{A}_{\delta}$ then we can find $m<\omega$ such
that $A_{\xi_{1}},...,A_{\xi_{n}}\in\mathcal{A}\left(  p\left(
g\upharpoonright m\right)  \right)  $ and then we know that $Sp\left(
p\left(  g\upharpoonright m+1\right)  ,f\right)  \nsubseteq^{\ast}A_{\xi_{1}%
}\cup...\cup A_{\xi_{n}}$ and since $Sp\left(  p\left(  g\upharpoonright
m+1\right)  ,f\right)  $ is contained in $im\left(  f\right)  $ we conclude
that $im\left(  f\right)  \in\mathcal{I}\left(  \mathcal{A}_{\delta}\right)
^{+}.$

\qquad\ \ 

Finally, find a partition $\left\{  Z_{n}\mid n\in\omega\right\}
\subseteq\mathcal{I}\left(  \mathcal{A}_{\delta}\right)  ^{+}$ of $im\left(
f\right)  $ and using the method of Shelah construct $A_{\delta+n}$ such that
$A_{\delta+n}\subseteq Z_{n}.$ This finishes the proof.
\end{proof}

\qquad\ \ \ 

\section*{More constructions}

In this last section, we will show the relationship between $+$-Ramsey and
other properties of \textsf{MAD }families. Recall than ideal $\mathcal{I}$ in
$\omega$ is \emph{tall }if for every $X\in\left[  \omega\right]  ^{\omega}$
there is $Y\in\mathcal{I}$ such that $X\cap Y$ is infinite. Note that if
$\mathcal{A}$ is an \textsf{AD }family then $\mathcal{I}\left(  \mathcal{A}%
\right)  $ is tall if and only if $\mathcal{A}$ is \textsf{MAD}. Note the
proof of theorem 23 in fact gives the following result:

\begin{corollary}
[$\mathfrak{s\leq a}$]If $\mathcal{I}$ is a tall ideal, then there is a
$+$-Ramsey \textsf{MAD }family $\mathcal{A}$ such that $\mathcal{A}%
\subseteq\mathcal{I}.$
\end{corollary}

\qquad\ \ 

The following are properties of \textsf{MAD }families that have been studied
in the literature:

\begin{definition}
Let $\mathcal{A}$ be a \textsf{MAD} family.

\begin{enumerate}
\item $\mathcal{A}$ is $\mathbb{P}$\emph{-indestructible} if $\mathcal{A}$
remains \textsf{MAD} after forcing with $\mathbb{P}$ (we are mainly interested
where $\mathbb{P}$ is Cohen, random, Sacks or Miller forcing).

\item $\mathcal{A}$ is \emph{weakly tight }if for every $\left\{  X_{n}\mid
n\in\omega\right\}  \subseteq\mathcal{I}\left(  \mathcal{A}\right)  ^{+}$
there is $B\in\mathcal{I}\left(  \mathcal{A}\right)  $ such that $\left\vert
B\cap X_{n}\right\vert =\omega$ for infinitely many $n\in\omega.$\qquad

\item $\mathcal{A}$ is\ \emph{tight }if for every $\left\{  X_{n}\mid
n\in\omega\right\}  \subseteq\mathcal{I}\left(  \mathcal{A}\right)  ^{+}$
there is $B\in\mathcal{I}\left(  \mathcal{A}\right)  $ such that $B\cap X_{n}$
is infinite for every $n\in\omega.$

\item $\mathcal{A}$ is \emph{Laflamme }if $\mathcal{A}$ can not be extended to
an $F_{\sigma}$-ideal.

\item $\mathcal{A}$ is $+$\emph{-Ramsey} if for every $\mathcal{I}\left(
\mathcal{A}\right)  ^{+}$-branching tree $T,$ there is $f\in\left[  T\right]
$ such that $im\left(  f\right)  \in\mathcal{I}\left(  \mathcal{A}\right)
^{+}.$
\end{enumerate}
\end{definition}

\qquad\ \ \qquad\ \ \qquad\ \qquad\ \ \ \ 

It is known that tightness implies both weak tightness and Cohen
indestructibility (see \cite{OrderingMADfamiliesalaKatetov}). It is also easy
to see that Cohen indestructibility implies Miller indestructibility and Sacks
indestructibility is weaker than both Miller indestructibility and random
indestructibility (see\cite{ForcingIndestructibilityofMADFamilies}).

\begin{corollary}
[$\mathfrak{s\leq a}$]There is a $+$-Ramsey \textsf{MAD }family that is not
Sacks indestructible, Laflamme or weakly tight.
\end{corollary}

\begin{proof}
The corollary follows by the previous result. In \cite{ForcingwithQuotients}
it was proved that there is a tall ideal $\mathcal{I}$ such that every
\textsf{MAD }family contained in $\mathcal{I}$ is Sacks destructible. A
similar result for weak tightness was proved in
\cite{CombinatoricsofMADFamilies}.
\end{proof}

\qquad\ \ \qquad\ \ 

The following is a very important definition:

\begin{definition}
We say $\varphi:\wp\left(  \omega\right)  \longrightarrow\omega\cup\left\{
\omega\right\}  $ is a \emph{lower semicontinuous submeasure} if the following hold:

\begin{enumerate}
\item $\varphi\left(  \omega\right)  =\omega.$

\item $\varphi\left(  A\right)  =0$ if and only if $A=\emptyset.$

\item $\varphi\left(  A\right)  \leq\varphi\left(  B\right)  $ whenever
$A\subseteq B.$

\item $\varphi\left(  A\cup B\right)  \leq\varphi\left(  A\right)
+\varphi\left(  B\right)  $ for every $A,B\subseteq X.$

\item (lower semicontinuity) if $A\subseteq\omega$ then $\varphi\left(
A\right)  =sup\left\{  \varphi\left(  A\cap n\right)  \mid n\in\omega\right\}
.$
\end{enumerate}
\end{definition}

\qquad\ \ \qquad\ \ 

Given a lower semicontinuous submeasure $\varphi$ we define \textsf{Fin}%
$\left(  \varphi\right)  $ as the family of those subsets of $\omega$ with
finite submeasure. The following is a very interesting result of Mazur:

\begin{proposition}
[Mazur \cite{Mazur}]$\mathcal{I}$ is an $F_{\sigma}$-ideal if and only if
there is a lower semicontinuous submeasure such that $\mathcal{I}=$
\textsf{Fin}$\left(  \varphi\right)  .$
\end{proposition}

\qquad\ \ 

If $a\subseteq\omega^{<\omega}$ we define $\pi\left(  a\right)  =\left\{
f\in\omega^{\omega}\mid\exists^{\infty}n\left(  f\upharpoonright n\in
a\right)  \right\}  .$ Given $f\in\omega^{\omega}$, define $\widehat
{f}=\left\{  f\upharpoonright n\mid n\in\omega\right\}  $ and let
$\mathcal{BR}=\{\widehat{f}\mid f\in\omega^{\omega}\}.$ By $\mathcal{J}$ we
denote the ideal on $\omega^{<\omega}$ consisting of all sets $a\subseteq
\omega^{<\omega}$ such that $\pi\left(  a\right)  $ is finite. Clearly
$\mathcal{BR}\subseteq\mathcal{J}.$ The next result follows easily from the
results in \cite{TesisDavid}, but we include a proof for the convenience of
the reader:

\begin{lemma}
$\mathcal{J}$ can not be extended to an $F_{\sigma}$-ideal.
\end{lemma}

\begin{proof}
Let $\varphi:\wp\left(  \omega\right)  \longrightarrow\omega\cup\left\{
\omega\right\}  $ be a lower semicontinous submeasure, we will prove that
$\mathcal{J}$ is not a subset of \textsf{Fin}$\left(  \varphi\right)  .$ Given
$s\in\omega^{<\omega}$ we denote $B_{0}\left(  s\right)  =\left\{  t\in
\omega^{<\omega}\mid s\subseteq t\right\}  $ and $B_{1}\left(  s\right)
=\left\{  t\in\omega^{<\omega}\mid s\perp t\right\}  $ (where $s\perp t$
denotes that $s$ and $t$ are incompatible). Let $\omega^{<\omega}=\left\{
s_{n}\mid n\in\omega\right\}  .$ We recursively construct two $\left\langle
i_{n}\mid n\in\omega\right\rangle $ and $\left\langle F_{n}\mid n\in
\omega\right\rangle $ such that for every $n\in\omega$ the following holds:

\begin{enumerate}
\item $i_{n}\in\left\{  0,1\right\}  .$

\item $F_{n}$ is a finite subset of $\omega$ and$\ \varphi\left(
F_{n}\right)  \geq n+1.$

\item $%
{\textstyle\bigcap\limits_{j\leq n}}
B_{i_{j}}\left(  s_{j}\right)  \in$ \textsf{Fin}$\left(  \varphi\right)
^{+}.$

\item $F_{n}\subseteq%
{\textstyle\bigcap\limits_{j\leq n}}
B_{i_{j}}\left(  s_{j}\right)  .$
\end{enumerate}

\qquad\ \ 

The construction is very easy to perform, let $G=%
{\textstyle\bigcup\limits_{n\in\omega}}
F_{n}.$ Note that $G\in$ \textsf{Fin}$\left(  \varphi\right)  ^{+}.$
Furthermore, for every $s\in\omega^{<\omega}$ either $G$ is almost contained
in $B_{0}\left(  s\right)  $ or is almost disjoint from it. It is easy to see
that $G\in\mathcal{J}$ so $\mathcal{J}$ is not contained in \textsf{Fin}%
$\left(  \varphi\right)  .$
\end{proof}

\qquad\ \ 

We can now prove the following:

\begin{proposition}
[$\mathsf{CH}$]There is a Laflamme \textsf{MAD} family that is not $+$-Ramsey.
\end{proposition}

\begin{proof}
Let $\left\{  \mathcal{I}_{\alpha}\mid\alpha\in\omega_{1}\right\}  $ be the
set of all $F_{\sigma}\,$-ideals in $\omega^{<\omega}.$ We construct
$\mathcal{A}=\left\{  A_{\alpha}\mid\alpha<\omega_{1}\right\}  $ such that the
following holds:

\begin{enumerate}
\item $\mathcal{A\cup BR}$ is an \textsf{AD} family.

\item If $s\in\omega^{<\omega}$ then $\mathcal{A}$ contains an infinite
partition of $\left\{  s^{\frown}n\mid n\in\omega\right\}  $ into infinite sets.

\item If\emph{ }$\mathcal{A}_{\alpha}\mathcal{\cup BR}\subseteq\mathcal{I}%
_{\alpha}$ then $A_{\alpha}\notin\mathcal{I}_{\alpha}$ (where $\mathcal{A}%
_{\alpha}=\left\{  A_{\xi}\mid\xi<\alpha\right\}  $).$\qquad$

\item $\mathcal{A}_{\alpha}$ is countable.
\end{enumerate}

\qquad\ \ \ \qquad\ \ \ 

At step $\alpha$ assume that $\mathcal{BR}\cup\mathcal{A}_{\alpha}%
\subseteq\mathcal{I}_{\alpha}$. Since $\mathcal{I}_{\alpha}$ is an $F_{\sigma
}$-ideal and it contains all branches, there is $a\in\mathcal{I}_{\alpha}%
^{+}\cap\mathcal{J}.$ Let $\pi\left(  a\right)  =\left\{  f_{1},...,f_{n}%
\right\}  $ and we now define $b=a\setminus(\widehat{f}_{1}\cup...\cup
\widehat{f}_{n}).$ Note that $\pi\left(  b\right)  =\emptyset$ and
$b\in\mathcal{I}_{\alpha}^{+}.$ Let $\varphi$ be a lower semicontinuous
submeasure such that $\mathcal{I}_{\alpha}=$ \textsf{Fin}$\left(
\varphi\right)  $ and let $\mathcal{A}_{\alpha}=\left\{  B_{n}\mid n\in
\omega\right\}  .$ We recursively find $s_{n}\subseteq b\setminus(B_{0}%
\cup...\cup B_{n})$ such that $\varphi\left(  s_{n}\right)  \geq n$ (this is
possible since $b\in\mathcal{I}_{\alpha}^{+}$). Then $A_{\alpha}=%
{\textstyle\bigcup\limits_{n\in\omega}}
s_{n}$ is the set we were looking for. It is easy to see that $\mathcal{A\cup
BR}$ is a Laflamme \textsf{MAD} family that is not $+$-Ramsey.
\end{proof}

\qquad\ \ 

We will now prove that weak tightness does not imply being $+$-Ramsey. Given
$s\in\omega^{<\omega}$ we define $\left[  s\right]  =\left\{  t\in
\omega^{<\omega}\mid s\subseteq t\right\}  .$

\begin{lemma}
If $A\subseteq\omega^{<\omega}$ does not have infinite antichains then $A$ can
be covered with finitely many chains.
\end{lemma}

\begin{proof}
Define $S$ as the set of all unsplitting nodes of $A$ i.e. $s\in A$ if and
only if every two extensions of $s$ in $A$ are compatible. Note that
$S\subseteq A$ and every element of $A$ can be extended to an element of $S$
(otherwise $A$ would contain a Sacks tree and hence an infinite antichain).
Let $B\subseteq S$ be a maximal (finite) antichain. For every $s\in B$ let
$b_{s}\in\omega^{\omega}$ the unique branch such that $A\cap\left[  s\right]
\subseteq\widehat{b}_{s}.$ Then (by the maximality of $B$) we conclude
$A\subseteq%
{\textstyle\bigcup\limits_{s\in B}}
\widehat{b}_{s}.$
\end{proof}

\qquad\ \ 

We need the following lemma:\qquad\ \ \ \ \qquad\ \ \ 

\begin{lemma}
If $A=\left\{  A_{n}\mid n\in\omega\right\}  \subseteq\wp\left(
\omega^{<\omega}\right)  $ is a collection of infinite antichains, then there
is an antichain $B$ such that $B\cap A_{n}$ is infinite for infinitely many
$n\in\omega.$
\end{lemma}

\begin{proof}
We say $s\in\omega^{<\omega}$ \emph{watches }$A_{n}$ if $s$ has infinitely
many extensions in $A_{n}.$ Define $T\subseteq\omega^{<\omega}$ such that
$s\in T$ if and only if there are infinitely many $n\in\omega$ such that $s$
watches $A_{n}.$ Note that $T$ is a tree. First assume there is $s\in T$ that
is a maximal node. By shrinking $A$ if needed, we may assume $s$ watches every
element of $A.$ We now define the set $C=\left\{  A_{n}\mid\exists^{\infty
}m\left(  A_{n}\cap\left[  s^{\frown}m\right]  \neq\emptyset\right)  \right\}
.$ In case $C$ is infinite, we can find an antichain $B$ that has infinite
intersection with every element of $C.$ Now assume that $C$ is finite, by
shrinking $A$ we may assume $C$ is the empty set. In this way, for every
$A_{n}$ there is $m_{n}$ such that $s^{\frown}m_{n}$ watches $A_{n}.$ We can
then find an infinite set $X\in\left[  \omega\right]  ^{\omega}$ such that
$m_{n}\neq m_{r}$ whenever $n\neq r$ and $n,r\in X$ (recall that $s$ is
maximal). Then $B=%
{\textstyle\bigcup\limits_{n\in X}}
\left[  s^{\frown}m_{n}\right]  \cap A_{n}$ is the set we were looking for.

\qquad\ \ \ \qquad\ \ 

Now we may assume $T$ does not have maximal nodes. If $T$ contains a Sacks
tree then we can find an infinite antichain $Y\subseteq T$. For every $s\in Y$
we choose $n_{s}$ such that $s$ watches $A_{n_{s}}$ and if $s\neq t$ then
$A_{n_{s}}\neq A_{n_{t}}.$ Then $B=%
{\textstyle\bigcup\limits_{s\in Y}}
\left[  s\right]  \cap A_{n_{s}}$ is the set we were looking for.

\qquad\ \ 

The only case left is that there is $s\in T$ that does not split in $T$ nor is
maximal. Let $f\in\left[  T\right]  $ the only branch that extends $s.$ We may
assum+e $s$ watches every element of $A$ and every $A_{n}$ is disjoint from
$\widehat{f}$ (this is because $A_{n}$ is an antichain and $f$ is a branch).
We say $A_{n}$ \emph{is a comb with }$f$ if $\Delta(A_{n}\cap\left[  s\right]
,\widehat{f})$ is infinite. We may assume that either every element of $A$ is
a comb with $f$ or none is. In case all of them are combs we can easily find
the desired antichain. So assume none of them are combs. In this way, for
every $n\in\omega$ we can find $t_{n}$ extending $s$ but incompatible with $f$
of minimal length such that $t_{n}$ watches $A_{n}.$ Since $t_{n}\notin T$ we
can find $W\in\left[  \omega\right]  ^{\omega}$ such that $t_{n}\neq t_{m}$
for all $n,m\in W$ where $n\neq m.$ Then we recursively construct the desired antichain.
\end{proof}

\qquad\ \ \ \ \ \ \ \ \ \ \ \ \ \ \qquad\ \ \ \ \ \ \ \ \ \ \ \ \ \ \ \ \ \ \ \ \ 

We can then conclude the following:\ \ \ \ \ \ \ \ \ \ \ \ \ \ \ \qquad\ \ \ \ \ \ \ \ \ \ 

\begin{proposition}
[$\mathsf{CH}$]There is a weakly tight \textsf{MAD} family that is not $+$-Ramsey.
\end{proposition}

\begin{proof}
Let $\left\{  \overline{X}_{\alpha}\mid\omega\leq\alpha<\omega_{1}\right\}  $
enumerate all countable sequences of infinite subsets of $\omega^{<\omega}.$
Let $\mathcal{BR}=\{\widehat{f}\mid f\in\omega^{\omega}\},$ we construct
$\mathcal{A}=\left\{  A_{\alpha}\mid\alpha<\omega_{1}\right\}  $ such that the
following holds:

\begin{enumerate}
\item Every $A_{\alpha}$ is an antichain.

\item $\mathcal{A\cup BR}$ is an \textsf{AD} family.

\item If $s\in\omega^{<\omega}$ then $\mathcal{A}$ contains a partition of
$suc\left(  s\right)  =\left\{  s^{\frown}n\mid n\in\omega\right\}  .$

\item For every $\omega\leq\alpha<\omega_{1}$ if $\overline{X}_{\alpha
}=\left\{  X_{n}\mid n\in\omega\right\}  \subseteq\mathcal{I}\left(
\mathcal{A}_{\alpha}\mathcal{\cup BR}\right)  ^{+}$\emph{ }then $A_{\alpha
}\cap X_{n}$ is infinite for infinitely many $n\in\omega$ (where
$\mathcal{A}_{\alpha}=\left\{  A_{\xi}\mid\xi<\alpha\right\}  $).
\end{enumerate}

\qquad\ \ \ \ \ \ \ \ 

At step $\alpha=\left\{  \alpha_{n}\mid n\in\omega\right\}  $ assume
$\overline{X}_{\alpha}=\left\{  X_{n}\mid n\in\omega\right\}  \subseteq\left(
\mathcal{A}_{\alpha}\mathcal{\cup BR}\right)  ^{+}.$ We first claim that there
is an infinite antichain $X_{n}^{\prime}\subseteq X_{n}$ such that $X_{n}%
\in\mathcal{A}_{\alpha}^{\perp}.$ Let $\Sigma=\left\{  A\in\mathcal{A}%
_{\alpha}\mid\left\vert A\cap X_{n}\right\vert =\omega\right\}  .$ In case
$\Sigma$ is finitem by lemma 32 we can find an infinit antichain
$X_{n}^{\prime}\subseteq X_{n}\setminus%
{\textstyle\bigcup}
\Sigma.$ If $\Sigma$ is infinite, then by lemma 33 we can find an infinite
$\Sigma^{\prime}\subseteq\Sigma$ and $B_{A}\in\left[  A\cap X_{n}\right]
^{\omega}$ for $A\in\Sigma^{\prime}$ such that $%
{\textstyle\bigcup}
\left\{  B_{A}\mid A\in\Sigma^{\prime}\right\}  $ is an antichain. It is then
easy to choose distinct $\left\{  s_{A}\in B_{A}\mid A\in\Sigma^{\prime
}\right\}  $ so that $X_{n}^{\prime}=\left\{  s_{A}\in B_{A}\mid A\in
\Sigma^{\prime}\right\}  \in\mathcal{A}_{\alpha}^{\perp}.$

Let $Y_{n}=X_{n}^{\prime}\setminus\left(  A_{\alpha_{0}}\cup...A_{\alpha_{n}%
}\right)  $ which is an infinite antichain. By the lemma 33 we can find an
antichain $A_{\alpha}$ $\subseteq%
{\textstyle\bigcup\limits_{n\in\omega}}
Y_{n}$ such that $A_{\alpha}\cap Y_{n}$ is infinite for infinitely many
$n\in\omega.$

\qquad\ \ \ \ 

Clearly $\mathcal{A\cup BR}$ is not $+$-Ramsey (recall that weakly tight
families are maximal).
\end{proof}

\qquad\ \ 

Recall that Miller indestructibility implies being $+$-Ramsey. We will now
prove that (in particular) Sacks or random indestructibility are not enough to
get $+$-Ramseyness. We will say a family $\mathcal{A}$ on $\omega^{<\omega}$
is a \emph{standard }$\mathcal{K}_{\sigma}$\emph{ family }if the following holds:

\begin{enumerate}
\item $\mathcal{A}$ is an \textsf{AD} family.

\item If $A\in\mathcal{A}$ either $\pi\left(  A\right)  =\emptyset$ or $A$ is
a finitely branching tree on $\omega^{<\omega}.$

\item If $s\in\omega^{<\omega}$ then $\left\{  s^{\frown}n\mid n\in
\omega\right\}  \in\mathcal{I}\left(  \mathcal{A}\right)  ^{++}.$
\end{enumerate}

\qquad\ \ \ 

Recall that if $a\subseteq\omega^{<\omega},$ we denoted $\pi\left(  a\right)
=\left\{  f\in\omega^{\omega}\mid\exists^{\infty}n\left(  f\upharpoonright
n\in a\right)  \right\}  .$ We now need the following lemma:

\begin{lemma}
Let $\mathbb{P}$ be an $\omega^{\omega}$-bounding forcing and $\mathcal{A}$ a
countable standard $\mathcal{K}_{\sigma}$ family. If $p\in\mathbb{P}$ and
$\dot{b}$ is a $\mathbb{P}$-name for an infinite subset of $\omega^{<\omega}$
such that $p\Vdash``\dot{b}\in\mathcal{A}^{\perp}\textquotedblright$ then
there are $q\leq p$ and $\mathcal{B}$ a countable standard $\mathcal{K}%
_{\sigma}$ family such that $\mathcal{A\subseteq B}$ and $q\Vdash``\dot
{b}\notin\mathcal{B}^{\perp}\textquotedblright.$
\end{lemma}

\begin{proof}
Let $\mathcal{A}=\left\{  T_{n}\mid n\in\omega\right\}  \cup\left\{  a_{n}\mid
n\in\omega\right\}  $ where $T_{n}$ is a finitely branching subtree of
$\omega^{<\omega}$ and $\pi\left(  a_{n}\right)  =\emptyset$ for every
$n\in\omega.$ We may assume that $p$ forces that $\pi(\dot{b})$ is either
empty or a singleton. We first assume there is $\dot{r}$ such that
$p\Vdash``\pi(\dot{b})=\left\{  \dot{r}\right\}  \textquotedblright.$ Since
$\mathbb{P}$ is $\omega^{\omega}$-bounding, we may find $p_{1}\leq p$ and
$T\in V$ a finitely branching well pruned subtree of $\omega^{<\omega}$ such
that $p_{1}\Vdash``\dot{r}\in\left[  T\right]  \textquotedblright.$ Once
again, since $\mathbb{P}$ is $\omega^{\omega}$-bounding we may find $p_{2}\leq
p_{1}$ and $f\in\omega^{\omega}$ such that the following holds:

\begin{enumerate}
\item $f$ is an increasing function.

\item $p_{2}\Vdash``\left(  T_{n}\cup a_{n}\right)  \cap\widehat{r}%
\subseteq\omega^{<f\left(  n\right)  }\textquotedblright.$
\end{enumerate}

\qquad\ \ \ 

For each $n\in\omega,$ define $\widetilde{T}_{n}=\left\{  s\in T_{n}\mid
f\left(  n\right)  \leq\left\vert s\right\vert \right\}  $ and define the set
$\widetilde{a}_{n}$ as $\left\{  t\mid\exists s\in a_{n}\left(  s\in
a_{n}\wedge f\left(  n\right)  \leq f\left(  n\right)  \right)  \right\}  .$
Let $K=T\setminus%
{\textstyle\bigcup\limits_{n\in\omega}}
(\widetilde{T}_{n}\cup\widetilde{a}_{n}).$ It is easy to see that $K$ is a
finitely branching tree, $p_{2}\Vdash``\dot{r}\in\left[  K\right]
\textquotedblright$ and $K\in\mathcal{A}^{\perp}.$ We now simply define
$\mathcal{B}=\mathcal{A\cup}\left\{  K\right\}  .$

\qquad\ \ 

Now we consider the case where $\pi(\dot{b})$ is forced to be empty. Let
$\dot{S}$ be the tree of all $s\in\omega^{<\omega}$ such that $s$ has
infinitely many extensions in $\dot{b}.$ We will first assume there are
$p_{1}\leq p$ and $s$ such that $p_{1}$ forces that $s$ is a maximal node of
$\dot{S}.$ Since $\mathbb{P}$ is $\omega^{\omega}$-bounding, we can find a
ground model interval partition $\mathcal{P}=\left\{  P_{n}\mid n\in
\omega\right\}  $ and $p_{2}\leq p_{1}$ such that if $n\in\omega$ then $p_{2}$
forces that there is $\dot{m}_{n}\in P_{n}$ such that $(\left[  s^{\frown}%
\dot{m}_{n}\right]  \cap\dot{b})\setminus\left(  T_{0}\cup...T_{n}\cup
a_{0}\cup...\cup a_{n}\right)  \neq\emptyset.$ Given $n,m\in\omega$ we define
$K_{n,m}=\left\{  s^{\frown}i^{\frown}t\mid i\in P_{n}\wedge t\in
m^{m}\right\}  .$ Using once again that $\mathbb{P}$ is $\omega^{\omega}%
$-bounding, we may find $p_{3}\leq p_{2}$ and an increasing function
$f:\omega\longrightarrow\omega$ such that if $n\in\omega$ then $p_{3}$ forces
$(K_{n,f\left(  n\right)  }\cap\dot{b})\setminus\left(  T_{0}\cup...T_{n}\cup
a_{0}\cup...\cup a_{n}\right)  $ is non-empty for every $n\in\omega.$ We now
define $a=%
{\textstyle\bigcup\limits_{n\in\omega}}
K_{n,f\left(  n\right)  }\setminus\left(  T_{0}\cup...T_{n}\cup a_{0}%
\cup...\cup a_{n}\right)  .$ It is easy to see that $\pi\left(  a\right)
=\emptyset,$ $a\in\mathcal{A}^{\perp}$ and $p_{3}$ forces that $a$ and
$\dot{b}$ have infinite intersection.

\qquad\ \ 

Now we assume that $p$ forces that $\dot{S}$ does not have maximal nodes, let
$\dot{r}$ be a name for a branch of $\dot{S}.$ First assume that $\dot{r}$ is
forced to be a branch through some element of $\mathcal{A}.$ We may assume
that $p\Vdash``\dot{r}\in\left[  T_{0}\right]  \textquotedblright.$ Since
$\mathbb{P}$ is $\omega^{\omega}$-bounding, we may find $p_{1}\leq p$ and an
increasing ground model function $f:\omega\longrightarrow\omega$ such that if
$n\in\omega$ then $p_{1}$ forces that all extentions of $\dot{r}%
\upharpoonright f\left(  n\right)  $ to $\dot{b}$ are not in $T_{0}%
\cup...T_{n}\cup a_{0}\cup...\cup a_{n}.$ Once again, we may find $p_{2}\leq
p_{1}$ and $g:\omega\longrightarrow\omega$ such that if $n\in\omega$ then
$\dot{b}$ has non empty intersection with the set $\{\dot{r}\upharpoonright
f\left(  n\right)  ^{\frown}t\mid t\in g\left(  n\right)  ^{g\left(  n\right)
}\}\setminus\left(  T_{0}\cup...T_{n}\cup a_{0}\cup...\cup a_{n}\right)  .$ We
now define $a=%
{\textstyle\bigcup\limits_{s\in\left(  T_{0}\right)  _{f\left(  n\right)  }}}
(\{s^{\frown}t\mid t\in g\left(  n\right)  ^{g\left(  n\right)  }%
\}\setminus\left(  T_{0}\cup...T_{n}\cup a_{0}\cup...\cup a_{n}\right)  ).$ It
is easy to see that $a$ has the desired properties.

\qquad\ \ \ \ \ \ \ 

Finally, in case that $\dot{r}$ is not forced to be a branch through some
element of $\mathcal{A},$ we find a finitely branching tree $T\in
\mathcal{A}^{\perp}$ such that $p\Vdash``\dot{r}\in\left[  T\right]
\textquotedblright$ as we did at the beginning of the proof. If $T$ has
infinite intersection with $\dot{b}$ we are done and if not then we apply the
previous case.
\end{proof}

\qquad\ \ \ \ 

With a standard bookkeeping argument we can then conclude the following:

\begin{proposition}
[$\mathsf{CH}$]If $\mathbb{P}$ is a proper $\omega^{\omega}$-bounding forcing
of size $\omega_{1},$ then there is a \textsf{MAD} family $\mathcal{A}$ that
is $\mathbb{P}$ indestructible but is not $+$-Ramsey.
\end{proposition}

\qquad\ \ 

\begin{acknowledgement}
I would like to thank my advisor Michael Hru\v{s}\'{a}k. Not only was him who
introduced me to the topic of this paper, but he has always been a great
support for me. The author would also like to thank the generous referee for
his or her valuable suggestions and corrections. The author would also like to
thank Jonathan Cancino and Arturo Mart\'{\i}nez.
\end{acknowledgement}

\bibliographystyle{plain}
\bibliography{Ramsey}

\qquad\ \ \ 

Osvaldo Guzm\'{a}n

York University

oguzman9@yorku.ca
\end{document}